\documentclass[12pt]{amsart}
\usepackage{amssymb}
\usepackage{latexsym}
\usepackage{amsrefs}
\usepackage{upref}

\swapnumbers

\theoremstyle{plain}
\newtheorem{blank}{}[section]

\newtheorem{prop}[blank]{Proposition}
\newtheorem{thm}[blank]{Theorem}
\newtheorem{cor}[blank]{Corollary}

\theoremstyle{definition}
\newtheorem{dblank}[blank]{}
\newtheorem*{defs}{Definitions}

\theoremstyle{remark}
\newtheorem*{rem}{Remark}

\newcommand{\ie}{\textit{i.e.}}
\newcommand{\eg}{\textit{e.g.}}
\newcommand{\cf}{\textit{cf.}}

\newcommand{\tx}{\textup}

\newcommand{\rest}{{\restriction}}

\newcommand{\set}[1]{\{\, #1\,\}}

\newcommand{\N}{\mathbb N}
\newcommand{\Q}{\mathbb Q}
\newcommand{\R}{\mathbb R}
\newcommand{\Z}{\mathbb Z}

\newcommand{\rr}{\mathfrak R}

\newcommand{\grint}[1]{\left\lfloor#1\right\rfloor}

\DeclareMathOperator{\Th}{Th}

\newcommand{\nbd}{\nobreakdash-\hspace{0pt}}

\makeatletter
   \def\th@plain{\slshape}
   \makeatother

\begin{document}
\title[Connectedness in structures on the real numbers]
{Connectedness in structures on the real numbers: o-minimality and undecidability}

\author[A.~Dolich]{Alfred Dolich}
\address{Department of Mathematics and Computer Science\\
Kingsborough Community College\\
2001 Oriental Boulevard\\
Brooklyn, New York 11235, USA}
\address{
Department of Mathematics\\
The Graduate Center\\
365 Fifth Avenue\\
Room 4208\\
New York, New York 10016, USA}
\email{alfredo.dolich@kbcc.cuny.edu}
 
\author[C.~Miller]{Chris Miller}
\address
{Department of Mathematics\\
Ohio State University\\
231 West~18th Avenue\\
Columbus, Ohio 43210, USA}
\email{miller@math.osu.edu}

\author[A.~Savatovsky]{Alex Savatovsky}
\address
{Department of Mathematics and Statistics\\
University of Konstanz\\
78457 Konstanz, Germany}
\email{alex@savatovsky.net}

\author[A.~Thamrongthanyalak]
{Athipat Thamrongthanyalak}
\address
{Department of Mathematics\\
Ohio State University\\
231 West~18th Avenue\\
Columbus, Ohio 43210, USA}
\curraddr
{Department of Mathematics and Computer Science\\
Faculty of Science\\
Chulalongkorn University\\
Bangkok\\
10330\\
Thailand}
\email{athipat.th@chula.ac.th}




\thanks{\today. Some version of this document has been submitted for publication. Comments are welcome. Miller is the corresponding author.}

\begin{abstract}
We initiate an investigation of structures on the set of real numbers having the property that path components of definable sets are definable.
All o\nobreakdash-\hspace{0pt}minimal structures on $(\mathbb{R},<)$ have the property, as do all 
expansions of $(\mathbb{R},+,\cdot,\mathbb{N})$. 
Our main analytic-geometric result is that any such expansion of $(\mathbb{R},<,+)$ by
boolean combinations of open sets (of any arities) either is o\nobreakdash-\hspace{0pt}minimal  or defines an isomorph of $(\mathbb N,+,\cdot\,)$.
We also show that any given expansion of $(\mathbb{R}, <, +,\mathbb{N})$ by subsets of $\mathbb{N}^n$ ($n$ allowed to vary)
has the property if and only if it defines all arithmetic sets.   
Variations arise by considering connected components or quasicomponents instead of path components.
 \end{abstract}

\maketitle

\section{Introduction}

We are interested in structures on the set of real numbers, 
especially expansions of the real field $(\R,+,\cdot\,)$ by locally path connected sets, 
having the property that each path component of each definable set (allowing parameters, and of any arity) is definable.
Variations on the theme arise by considering connected components and quasicomponents.
Our original motivation lies in real-analytic geometry,
but we have come to believe that a more foundational investigation is in order.
Thus, in this paper, we provide some fundamental logical results and some examples.
 
\subsection*{Some global conventions}
The reader is assumed to be familiar with basic topology (\eg, Munkres~\cite{munkres}) and mathematical logic (especially o-minimality, \eg, van den Dries~\cite{ominbook}).
Our default is to identify interdefinable structures; 
we shall make it clear when dependence on the language is relevant.
We always regard each $\R^n$ as equipped with its usual topology, and subsets of $\R^n$ with the induced topology.
We abbreviate ``connected component'' by ``component''.
The set of all nonnegative integers is denoted by $\N$.

\medskip

Proofs of results stated in this introduction are postponed to later in the paper.

\begin{defs}
A structure with underlying set $\R$ is:
\begin{enumerate}
\item
\textbf{path-component closed}, or \textbf{PCC} for short,
if for every definable set $E$, each path component of $E$ is definable;
\item
\textbf{component closed}, or \textbf{CC} for short,
if \tx{(1)} holds with ``component'' in place of ``path component'';
\item
\textbf{quasicomponent closed}, or \textbf{QCC} for short,
if \tx{(1)} holds with ``quasicomponent'' in place of ``path component''.
\end{enumerate}
Note that every structure on $\R$ that satisfies any of these conditions
defines the usual order $<$ on $\R$. (Consider $\set{(x,y)\in\R^2: x\neq y}$.) 
Thus, the definitions are really about expansions of $(\R,<)$. 
\end{defs}

We suspect that, in this generality, none of these conditions imply any of the others, but our candidates for counterexamples are ad hoc and unverified as yet.
Currently we are most interested in PCC structures, but simultaneous consideration of the other two conditions is natural.

Trivially, the expansion of $\R$ by each subset of $\R^n$ as $n$ ranges over the positive integers is PCC, CC and QCC.
Thus, abstractly, every structure on $\R$ has a least expansion that is PCC, CC and QCC; 
of course, this holds also for each of the conditions individually and pairwise.
It is easy (relative to currently available technology) to see that there exist two concrete classes of nontrivial examples.

\begin{prop}\label{extremes}
Every o\nbd minimal expansion of $(\R,<)$ is PCC\tx, CC and QCC.
Every expansion of $(\R,+,\cdot,\N)$ is PCC\tx, CC and QCC.
\end{prop}

O\nbd minimal expansions of  $(\R,<)$ have been studied extensively. 
Expansions of $(\R,+,\cdot,\N)$ are too complicated for us to study as definability theory (every real projective set is definable).
The main underlying reason for starting this investigation was to see if certain properties of o\nbd minimal expansions of $(\R,+,\cdot\,)$ extend to some other settings.
Thus: \emph{Do there exist PCC structures on the real field that are generated by locally path connected sets, but neither define $\N$ nor are o\nbd minimal?}
Currently, we do not know.
However, we do have a significant partial result, as we now explain.

\subsubsection*{Notation}
Throughout,  $\mid$ indicates the usual divisibility relation on $\N$, regarded as  the set $\set{(n,dn): n,d\in \N}$. 

\medskip

By Robinson~\cite{MR0031446} , $(\N,<,\mid\,)$ is interdefinable with $(\N,+,\cdot\,)$, hence also with 
the expansion of $\N$ by all arithmetic sets (of any arity); we tend to use this fact without further mention.

\begin{thm}\label{dichot}
If $\rr$ is an expansion of $(\R,<,+)$ by boolean combinations of open subsets of various $\R^n$\tx, and is
PCC\tx, CC or QCC\tx,  then one of the following holds.
\begin{enumerate}
\item
$\rr$ is o\nbd minimal.
\item
There is a strictly increasing $\sigma\colon\N\to\R$ such that $\set{(\sigma(m),\sigma(n)): m \mid n}$ is definable.
\end{enumerate}
\end{thm}

As we shall see later,  every expansion of $(\R,<,+,\mid\,)$ by subsets of various $\N^n$ is PCC, CC and QCC.
There are expansions of $(\R,+,\cdot)$ for which condition~(2) holds and every definable set is a finite union of locally path connected definable sets---for some examples, see
Friedman and Miller~\cites{ominsparse,fast},
\cite{tameness},
Miller and Tyne~\cite{itseq},
and~\cite{MR3800759}---but 
we do not yet know whether any of these examples are PCC, CC or QCC.

\begin{cor}
If $\rr$ is a PCC expansion of $(\R,<,+)$ generated by topological submanifolds\tx,
and $\rr$ is not o\nbd minimal\tx, then no expansion of $\rr$ has a decidable theory. 
\end{cor}

This is perhaps not so surprising, but also not as discouraging as it might seem.
Definability is generally more important than abstract interpretability in the analytic geometry of expansions of $(\R,<)$.

\pagebreak

Theorem~\ref{dichot} relies on the following technical result.
\begin{prop}\label{maintech}
If $\emptyset \neq A\subseteq \R$ has empty interior and 
$f\colon A\to A$ is strictly increasing and strictly dominates the identity\tx, 
then there is a strictly increasing $\sigma\colon\N\to\R$ such that every expansion of $(\R,<,f)$
that is PCC\tx, CC or QCC defines $\set{(\sigma(m),\sigma(n)): m \mid n}$.
 \end{prop}
 (By ``strictly dominates the identity'' we mean that $x<f(x)$ for all $x$ in the domain of $f$.
 For logicians: The function $f$ is regarded as a subset of $\R^2$, not the interpretation of a unary function symbol.)  
 A key step in the proof is the special case that $A=\N$ and $f$ is the usual successor function on $\N$.

\begin{prop}\label{definesmidintro}
Every expansion of $(\R,<,\N)$ that is PCC\tx, CC or QCC defines $\mid$\,.
\end{prop}

Via quantifier elimination in an appropriate language, it is easy to see that 
$(\R,<,\N)$ does not define $\set{2n: n\in\N}$, and so $(\R,<,\N)$ is neither PCC, CC nor QCC.

It appears to us at this time that dealing with arbitrary expansions of $(\R,<,+)$, especially those that define a bijection between a bounded interval and an unbounded interval, is daunting. 
Thus, we first attempt to understand some less complicated structures.
 
\textbf{From now on:} $\rr$ denotes a fixed, but arbitrary, expansion of $(\R,<)$.

We say that $\rr$ is \textbf{locally o\nbd minimal} if for each definable $E\subseteq \R$ and $x\in \R$ there is an open neighborhood $U$ of $x$ such that $E\cap U$ is a finite union of points and open intervals.\footnote{One might find different definitions in the literature, but as far as we know, all of them are equivalent over $(\R,<)$.}
Evidently, every o\nbd minimal expansion of $(\R,<)$ is locally o\nbd minimal, but there are locally o\nbd minimal expansions of $(\R,<)$ that are not o\nbd minimal. 
Of particular importance in this paper is that the expansion of $(\R,<,+)$ by all subsets of each $\N^n$ (as $n$ ranges over positive integers),  is locally o\nbd minimal 
(see, \eg, \cite{ominsparse}), 
and so the same is true of all of its reducts. 
Thus, $(\R,<,\N)$ is locally o\nbd minimal, but neither PCC, CC nor QCC, 

\begin{prop}\label{localominequivs}
The following are equivalent.
\begin{enumerate}
\item
$\rr$ is locally o\nbd minimal.
\item
If $E\subseteq \R^n$ is bounded and definable\tx, then the expansion of $(\R,<)$ by all definable subsets of $E$ is o\nbd minimal.
\item
Every definable set is locally path connected.
\item
For every definable set\tx, its components are the same as its path components and its quasicomponents.
\item
Components of bounded definable subsets of $\R^2$ are path connected.
\item
Quasicomponents of bounded definable subsets of $\R^2$ are connected.
\item
If $A\subseteq \R$ is definable and has empty interior\tx, then $A$ is closed and discrete.
\end{enumerate}
\end{prop}

\begin{cor}\label{localomincor}
If $\rr$ is locally o\nbd minimal\tx, then it is PCC if and only if it is CC\tx, if and only if it is QCC. 
\end{cor}

Thus, when dealing with locally o\nbd minimal expansions of $(\R,<)$, we may work with any of the conditions (PCC, CC or QCC) as convenient.  
 
 \pagebreak 

Next is our most interesting class of examples to date.
 
\begin{prop}\label{main}
Let $\rr$ be an expansion of $(\R,<,+,\mid\, )$ by subsets of various $\N^n$.
\begin{enumerate}
\item
Every bounded definable set is definable in $(\R,<,+)$.
\item
If $E$ is definable and connected\tx, then for all $x,y\in E$ there is a path in $E$ definable in $(\R,<,+)$ joining $x$ to $y$. 
\item
$\rr$ is CC.
\end{enumerate}
 \end{prop}
 
 \subsubsection*{Remarks}
(a)~Thus, $\rr$ is not just locally o\nbd minimal, but ``locally o\nbd minimal with respect to $(\R,<,+)$''.
(b)~Trivially, the theory of the expansion  of $\R$ by all subsets of each $\N^n$ interprets every countable theory. Thus, in contrast to its local geometry, the model theory of $\rr$ can be quite wild.
(c)~The group structure is not necessary for either~(1) or~(3) to hold, but~(2) can fail. 
To illustrate, it is not hard to check (or see~\cite{ccparis}) that the expansion  of $(\R,<)$ by all subsets of each $\N^n$ is CC, and every definable subset of any $(-\infty,b)^n$ with $b\in\R$ is definable in $(\R,<)$. 
But there is no definable path joining any point in $(-\infty,1)\times (-\infty,2)$ to the point $(1,2)$. 
 
\medskip
 
Of course, $\N$ is $\emptyset$\nbd interdefinable with $\Z$ over $(\R,<,+)$.
By combining results~\ref{definesmidintro}, \ref{localomincor} and \ref{main}, 
we obtain the following characterization. 

\begin{thm}\label{newmainthm}
If $\rr$ is an expansion of $(\R,<,+,\Z)$ by subsets of various $\Z^n$\tx,
then $\rr$ is PCC\tx, CC or QCC if and only if it defines all arithmetic sets. 
\end{thm}
 
We now say more about the original motivation for this work. 
One might be interested in understanding ``reachability'' properties of given $E\subseteq \R^n$. 
This can be made more precise by considering what is known if $(\R,+,\cdot,E)$ is o\nbd minimal. 
Since path components of $E$ are definable, it suffices to consider the case that $E$ is path connected; then it is path connected in a strong sense:
(a)~there is a definable map $\gamma\colon E^2\times [0,1]\to E$ such that, for all $x,y\in E$, the path $t\mapsto \gamma(x,y,t)\colon [0,1]\to E$ is rectifiable and joins $x$ to $y$; 
(b)~if moreover $E$ is compact, then there is a definable function $f\colon [0,\infty)\to [0,\infty)$ 
such that, for all $x,y\in E$, the length of $t\mapsto \gamma(x,y,t)$ is bounded by $f(\lVert x-y\rVert)$. 
(See van den Dries and Miller~\cite{geocat}*{4.21} for details.)
These are desirable properties in real-analytic geometry and control theory, and we hope that at least something similar will hold for certain cases that $(\R,+,\cdot,E)$ is not o\nbd minimal.
Here is a case of particular interest. 
Let $\omega$ be a nonzero real number, $S$ be the logarithmic spiral 
$
\set{e^t\cos(\omega t),e^t\sin(\omega t): t \in \R}
$,
and $E$ be definable in $(\R,+,\cdot,S)$.
Evidently,  $(\R,+,\cdot,S)$ is not locally o\nbd minimal, but it is known to have a number of desirable analytic-geometric properties; 
see, \eg, \cite{linear}*{4.4}, \cite{cpdmin}, and Tychonievich~\cite{tychthesis}*{Theorem~4.1.1}.
Some of these properties, in combination with Theorem~\ref{dichot}, reveal that $(\R,+,\cdot,S)$ is not PCC, but perhaps there might be PCC expansions of  $(\R,+,\cdot,S)$ in which many of these good properties, or variants thereof, are preserved.
See  also 
Aschenbrenner and Thamrongthanyalak~\cite{MR3365747},
Lafferriere and Miller~\cite{unifreach}, 
Sokantika and Thamrongthanyalak~\cite{MR3731015}, 
and~\cite{MR3896056}
for some related material on reachability properties of sets definable in certain expansions of $(\R,<)$.

The previous paragraph explains somewhat why we have so far mentioned working only over $(\R,<)$. 
But the question does arise: 
What about expansions of other dense linear orders, or even more generally,  first-order topological structures as defined in Pillay~\cite{fotop}?  
 We can do this to some extent, but the definitions of CC and QCC need to be modified, and it is unclear what do about PCC. 
To illustrate,  as $\Q$ is totally disconnected in the order topology, every structure on $\Q$ is trivially PCC, CC and QCC.
We propose that a reasonable analogue of CC in these more abstract settings is that every definable set be a union of maximal definably connected subsets. 
A priori, this property is weaker than CC for expansions of $(\R,<)$, but it is strong enough to obtain that every expansion of $(\R,<,\N)$ having the property defines $\mid$\, 
(\cf\ Proposition~\ref{definesmidintro}).
We defer further discussion of these more general settings to later in the paper. 
 
\subsubsection*{Some history and attributions}
The original questions underlying this research were raised by Thamrongthanyalak in conversations with Miller; some early results, including some first steps toward 
Theorem~\ref{newmainthm}, were documented in~\cite{ccparis}.
(See also Fornasiero ~\cite{MR2896830}.)
Savatovsky became interested in the project after attending a presentation by Miller at Universit\"at Konstanz in May 2017, subsequently leading to the establishment of Proposition~\ref{definesmidintro} and the motivating case (namely, $\rr=(\R,<,+,\mid\,)$) for Proposition~\ref{main}.
The proof presented here of Proposition~\ref{definesmidintro}
is due to Miller and  Thamrongthanyalak.
The statement and proof of Proposition~\ref{main} presented here is due essentially to Dolich. 
 
\subsubsection*{Acknowledgements} 
Research of Miller was supported in part by a Simons Visiting Professorship,
hosted by Mathematisches Forschungsinstitut Oberwolfach, Universit\"{a}t Konstanz, and Universit\'{e} Savoie Mont Blanc (Bourget-du-Lac).
Research of Thamrongthanyalak was conducted in part while a Zassenhaus Assistant Professor at Ohio State University.
We thank David Marker, Russell Miller, Serge Randrianbololona, Charles Steinhorn 
and Erik Walsberg for helpful observations.

\section{Proofs}
We shall not be providing proofs of results stated in the introduction in the same order as they were stated.

We begin with some conventions and routine facts from topology.
Let $Y$ be a topological space and $X\subseteq Y$.
The closure of $X$ is denoted by $\overline X$.
We say that $X$ is locally connected if, in the subspace topology, $X$ has a  basis consisting of connected sets. 
 If $u,v\in X$, then a path in $X$ from $u$ to $v$ is a continuous 
map $\gamma$ from a compact interval $[a,b]\subseteq \R$ into $Y$ such that $\gamma(a)=u$, $\gamma(b)=v$, and 
$\gamma([a,b])\subseteq X$. 
The set $X$ is path connected if for every $u,v\in X$ there is a path in $X$ from $u$ to $v$, and is locally path connected if, in the subspace topology, $X$ has a  basis consisting of path-connected sets.

\begin{dblank}\label{topfacts}
Here are some facts to keep in mind. 
\begin{itemize}
\item
The component of $y\in Y$ is the union  of all connected sets containing $y$ (and similarly for path components). 
\item
The quasicomponent of $y\in Y$ is the intersection of all clopen sets containing $y$.
\item
If $X$ is a component or a quasicomponent of $Y$, then $X$ is closed. \item
If $Y$ is locally connected, 
then its quasicomponents and its components are the same.
 \item
If $Y$ is locally path connected\tx, then its path components and its components are the same (which in turn are the same as its quasicomponents).
\item
If $Y$ has a countable basis at $y\in Y$ consisting of connected sets and $Y\setminus \{y\}$ is locally path connected, then $Y$ is locally path connected.
\end{itemize}
\end{dblank}

\begin{rem}
There are closed subsets of $\R^2$ having quasicomponents that are not components, \eg, 
$\{-1,1\}\times \R$ is a quasicomponent of the closure of the union of the boundaries of the boxes 
$
(-1+2^{-n}, 1-2^{-n})\times (-2^n,2^n)$, $n\in\N$.
\end{rem}
 
\begin{proof}[Proof of Proposition~\ref{extremes}]
Suppose that $\rr$ is an expansion of $(\R,+,\cdot,\N)$. We show that $\rr$ is PCC, CC and QCC.
Let $n\in\N$. 
Recall (see, \eg, \cite{ominbook}*{page~16}) that there exist $m\in \N$ and $Z\subseteq \R^m\times \R^n$ such that $Z$ is $\emptyset$\nbd definable in $(\R,+,\cdot,\N)$ and
$$
\bigl\{\, \set{y\in \R^n: (a,y)\in Z}: a\in \R^m\,\bigr\}
$$
is exactly the collection of the closed subsets of $\R^n$.
Let $E\subseteq \R^n$.
If $X$ is either a component or a quasicomponent of $E$, then $X$ is closed in $E$, and so 
$(\R,+,\cdot,\N,E)$ defines $X$.
Let $X$ be a path component of $E$; we show that $(\R,+,\cdot,\N,E)$ defines $X$.
It suffices to fix $x\in X$ and show that the set of all $v\in \R^n$ for which there is a path in $E$ from $x$ to $v$ is definable in $(\R,Z,E)$. 
Now, $v\in X$ if and only if there is a continuous $\gamma\colon [0,1]\to \R^n$ such that  $\gamma(0)=x$, $\gamma(1)=v$, and the image
of $[0,1]$ under $\gamma$ is contained in $E$.
To put this another way, 
$v\in X$ if and only if there exists $a\in \R^m$ such that: the fiber $Z_a:=\set{y\in \R^n: (a,y)\in Z}$ is a  bounded subset of $E$; 
for all $t\in [0,1]$ there exists a unique $y\in\R^n$ such that $(t,y)\in [0,1]\times Z_a$; and, the points $(0,x)$ and $(1,v)$ belong to $[0,1]\times Z_a$. (The resulting function is continuous because its graph is closed.)
As the set of all such $v$ is definable in $(\R,+,\cdot,\N,E)$, so is $X$. 

Suppose that $\rr$ is o\nbd minimal. 
 We show that $\rr$ is PCC, CC and QCC.
Every definable set has only finitely many components, each of which is definable~\cite{ominbook}*{Chapter 3, (2.18) and (2.19).7}. 
Hence, $\rr$ is CC. 
For PCC and QCC, it suffices by the basic topological facts to show that every definable set is locally path connected.
By cell decomposition and a routine induction on dimension, it suffices to let $C\subseteq \R^n$ be a cell of $\rr$ and $y\in\overline{C}\setminus C$, and show that  $\{y\}\cup C$ is locally path connected.
It is an exercise to see that  $\{y\}\cup C$ is locally connected at $y$ (indeed, 
see~\cite{ominbook}*{Chapter 3, (2.19).8}), that is, in the subspace tolopogy, $\{y\}\cup C$ has a basis at $y$ consisting of connected sets; 
as we are working over $\R$, we may take this basis to be countable. 
Now,  $(\{y\}\cup C)\setminus \{y\}=C$, and cells are locally path connected. 
By the last item of~\ref{topfacts},   $\{y\}\cup C$ is locally path connected at $y$ (as was to be shown).
\end{proof}
  
 \begin{proof}[Proof of Proposition~\ref{localominequivs}]
 This is mostly just an exercise, so we give only an outline.

(1)$\Rightarrow$(2).
If $\rr$ is locally o\nbd minimal, then it follows from the 
Bolzano-Weierstrass Theorem that 
every bounded definable subset of $\R$ is a finite union of points and open intervals.
Hence, if $E\subseteq \R^n$ is bounded, then there is an open interval $I\subseteq \R$ such that the expansion of $(I,<)$ by all definable (in $\rr$) subsets of $E$ is o\nbd minimal.  
It is now routine to see that the expansion of $(\R,<)$ by all definable
subsets of $E$ is also o\nbd minimal.  
 
(2)$\Rightarrow$(3). Recall the proof of Proposition~\ref{extremes}.

(3)$\Rightarrow$(4) is just topology, and (4)$\Rightarrow$((5) \& (6)) is trivial. 

(5)$\Rightarrow$(7).
If $E\subseteq \R$ has  empty interior and a limit point~$c$, then
$$
(E\setminus\{c\}\times [0,1])\cup \{(c,1)\}\cup (\R\times \{0\})
$$
is connected, but not path connected.
 
(6)$\Rightarrow$(7).
If $E\subseteq \R$ has empty interior and a limit point~$c$, then $\{(c,0),(c,1)\}$ is a quasicomponent of 
its union with $(E\setminus\{c\})\times [0,1]$.

(7)$\Rightarrow$(1).  
The boundary of every definable subset of $\R$ is closed and discrete. 
\end{proof}
 
 \begin{blank}\label{uniflocominpath}
Suppose there is an o\nbd minimal expansion $\rr_0$ of $(\R,<,+)$ such that every bounded set definable in $\rr$ is definable in $\rr_0$.
Let $E\subseteq \R^n$ be connected and definable in $\rr$.
Let $x,y\in E$. 
Then there is a path in $E$ definable in $\rr_0$ joining $x$ to $y$.  
\end{blank}

\begin{proof}
As $\rr$ is locally o\nbd minimal, $E$ is path connected. 
As images of paths are compact, there exist $N\in\N$ and a path in $E\cap [-N,N]^n$ joining $x$ to $y$. 
By assumption, $E\cap [-N,N]^n$ is definable in $\rr_0$. 
By o\nbd minimality, the path components of $E\cap [-N,N]^n$ are definable in $\rr_0$. 
Thus, we may reduce to the case that $E\cap [-N,N]^n$ is path connected. 
By Cell Decomposition and Curve Selection, there is a path in $E\cap [-N,N]^n$ joining $x$ to $y$ that is definable in $\rr_0$.
\end{proof}

For longer-range strategic reasons, we now introduce some greater generality.
Given a first-order topological structure $\mathfrak{M}$ with underlying set $M$, a set $X\subseteq M^n$ is said to be \textbf{definably connected} (with respect to $\mathfrak{M}$) if it is definable and not a union of two disjoint nonempty definable subsets of $X$ that are open in $X$. 
(To be clear, definably connected sets are definable \emph{by definition}.)  
Evidently, every connected definable set is definably connected, but the converse tends to fail in general.
It is not necessarily true that every definable set is a union of maximal definably connected subsets of $X$
(though this is true for all o\nbd minimal structures).  
 
\begin{prop}\label{definesmidQ}
If $\mathfrak{Q} $ is an expansion of $(\Q,<,\N)$ such that every definable set is a union of maximal definably connected sets\tx, then $\mathfrak{Q} $ defines $\mid$\, .
\end{prop}

\begin{proof}
For $d\in\N$, put $d\N:=\set{dn: n\in\N}$.
As  $a\mid b$ if and only if $b\N\subseteq a\N$, it suffices to show that 
$\bigcup_{d\in\N}(d\N\times \{d\})$ is definable.

First we show that each $d\N$ is definable.
Let $\mathbb S_0$ be the intersection of the boundary of $(0,\infty)^3$ with the union of the boundaries of the boxes $(-n,n)^3$, $0<n\in\N$.
Put $P=\set{(x,0,z)\in\Q^3: x,z>0}$.
For $d\in \N$, put
\[
\mathbb{S}_d=\bigl(\mathbb S_0\setminus P\bigr) \cup \bigl((\mathbb S_0\cap P)+(d,0,0)\bigr)
\cup \bigl([0,d]\times \{0\}\times \N^{>0}\bigr).
\]
In words, $\mathbb{S}_d$ is the result of replacing $\mathbb S_0\cap P$ by shifting it $d$ units in the positive $x$ direction, and then filling in the resulting gaps with the line segments $[0,d]\times \{0\}\times \{n\}$, $n>0$.
Each $\mathbb S_d$ is definable in $(\Q,<,\N)$. 
Fix $d$ and let $C_d$ be the union of all definably connected subsets of $\mathbb{S}_d$ that contain the point $(d,0,0)$. 
By assumption, $C_d$ is definable; by inspection, its intersection
with the positive $x$\nbd axis is the set $\set{(d(n+1),0,0):n\in\N}$.
Hence, each $d\N$ is definable. 

We now show that $+\rest \N^2$ (that is, addition on $\N$) is definable. 
We first define sets $\varGamma_{m,n} \subseteq \Q^6$ for $m,n\in\N$ by case distinction on the parities of $m$ and $n$.  
For $m,n\in2\N$, let $\varGamma_{m,n}$ be the union of the line segments in $\Q^6$ connecting, in order, 
the points 
$$
(m,n,0,0,0,0), (m,n,0,1,0,1), (m+1,n,0,1,0,1), (m+1,n+1,0,1,0,1).
$$
Note that $\varGamma_{m,n}$ is the union of the sets 
\begin{gather*}
\set{(m,n,0,t,0,t): 0\leq t\leq 1}\\
\set{(t,n,0,1,0,1): m\leq t\leq m+1}\\
\set{(m+1,t, 0,1,0,1): n\leq t\leq n+1}.
\end{gather*}
For $m,n\in2\N+1$, use the points 
$$
(m,n,0,0,0,0),  (m,n,0,1,0,1),  (m+1,n,0,1,0,1), (m+1,n+1,0,1,0,1). 
$$
For $m\in 2\N$ and $n\in2\N+1$, 
use the points 
$$
(m,n,0,0,0,0), (m,n,0,1,1,0), (m+1,n,0,1,1,0), (m+1,n+1,0,1,1,0).
$$ 
Finally, for $m\in2\N+1$ and $n\in2\N$, use the points
$$
(m,n,0,0,0,0), (m,n,1,0,0,1), (m+1,n,1,0,0,1), (m+1,n+1,1,0,0,1).
$$
Let $X$ be the union of $\N\times \bigcup_{m,n\in\N}\varGamma_{m,n}$  with the sets 
\begin{gather*}
\set{(d,0,d,t,t,t,t): d\in\N\And 0\leq t\leq 1}\\
\set{(t,0,t,1,1,1,1): t\geq 0}.
\end{gather*} 
Note that $X$ is definable in $(\Q,<,\N,2\N)$.
Let $C$ be the union of all definably connected subsets of $X$ that contain the origin (in $\Q^7)$.
By assumption, $C$ is definable. 
It is routine to see that the intersection of $C$ with $\Q^3\times \{0\}^4$ is equal to 
$+\rest \N^2\times \{0\}^4$.
Hence, $+\rest \N^2$ is definable. 

We are ready to finish the proof. 
As $+\rest \N^2$ is definable, shifting to the right by $d$ units is uniformly definable, and so  
$\bigcup_{d\in\N}(\mathbb S_d\times \{d\})$ is definable. 
Put $\ell=\set{(1,0,0,t): t\in \Q}$ and let $C'$ be the union of all definably connected subsets of
$
\ell\cup \bigcup_{d\in\N}( \mathbb S_d\times \{d\})
$
that contain the point $(1,0,0,1)$.
Observe that 
\[
C'\cap \bigl([0,\infty)\times \{0\}\times\{0\}\times \N\bigr)=\bigcup_{d\in \N}\set{(dn+1,0,0,d): d\in\N}. 
\]
Hence, $\bigcup_{d\in\N}(d\N\times \{d\})$ is definable,
as was to be shown.
\end{proof}
 
\begin{proof}[Proof of Proposition~\ref{definesmidintro}] This is a direct corollary of the proof of~\ref{definesmidQ}:
By repeating the proof over $(\R,<)$, we see that the sets $C$ and $C'$ are each components,  quasicomponents and path components of certain definable sets.
\end{proof}

Here is another corollary of the proof  of~\ref{definesmidQ}.

\begin{blank}[\cf~Proposition~\ref{definesmidintro}]
If $\rr$ is an expansion of $(\R,<,\N)$ such that every definable set is a union of maximal definably connected sets\tx, then $\rr$ defines $\mid$\, .
\end{blank}
 
\begin{proof}[Proof of Theorem~\ref{maintech}]
Let $\emptyset \neq A\subseteq \R$ have empty interior and 
$f\colon A\to A$ be strictly increasing and strictly dominating the identity.
Note that $A$ cannot have a maximal element. 
We must find a strictly increasing $\sigma\colon\N\to\R$ such that every expansion of 
$(\R,<)$ that defines $f$ and is PCC\tx, CC or QCC also defines $\set{(\sigma(m),\sigma(n)): m \mid n}$.
There exists $a_0\in A$ such that $A\cap [a_0,\infty)$ is infinite. 
By replacing $A$ with $A\cap [a_0,\infty)$ and $f$ with its restriction to  $A\cap [a_0,\infty)$ we reduce to the case that 
$\min A$ exists; for ease of notation, say $\min A=0$.
We obtain $S\subseteq \R^3$ by repeating the construction of $\mathbb S_1$ 
(as in the proof of~\ref{definesmidQ}), but with $\R$ instead of $\Q$, $A$ instead of $\N$, and $f$ instead of the successor function on $\N$.
Note that $S$ is definable from $f$ over $(\R,<)$. 
Let $P$ be the path component of $S$ that contains $f(0)$. 
Note that $P$ is a component and a quasicomponent of $S$, and the intersection of $P$ with $[0,\infty)\times \{0\}\times \{0\}$ is the orbit of $f$ on $f(0)$.
Thus, it suffices to assume that $A$ is the image of a strictly increasing function $\N\to\R$ and $f$ is the successor function on $A$; 
then $\bigl((-\infty,\sup A),<,A\bigr)$ is isomorphic to $(\R,<,\N)$ and we apply Proposition~\ref{definesmidintro} to finish. 
\end{proof}

\begin{proof}[Proof of Theorem~\ref{dichot}]
Let $\rr$ be an expansion of $(\R,<,+)$ by boolean combinations of open sets.
Suppose that $\rr$ is PCC, CC or QCC, but is not o\nbd minimal. 
We show that there is a strictly increasing $\sigma\colon \N\to\R$ such that $\set{(\sigma(m),\sigma(n)): m \mid n}$ is definable.
It suffices by the proof of Theorem~\ref{maintech} to show that  $\rr$ defines the range of a strictly increasing sequence in $\R$.
By Dougherty and Miller~\cite{DoM} and  
Dolich, Miller and Steinhorn~\cite{dms1},
$\rr$ defines an infinite discrete $A\subseteq \R$. 
 If $A$ is closed, then at least one of $A\cap [0,\infty)$ or $-A\cap [0,\infty)$ is the range of a strictly increasing sequence.  
Suppose now that $A$ is not closed.   
Put $F=\overline{A}\setminus A$; then $F$ is nonempty, has no interior, and is closed (because $A$ is locally closed).  

Suppose that $F$ has no isolated points; then there exists $N\in\N $ such that
$F\cap [-N,N]$ is a Cantor set.
The set of negatives  of lengths of the maximal open intervals of $[-N,N]\setminus F$ is definable and the range of a strictly increasing sequence of real numbers.
 
Suppose that $F$ has an isolated point $c$.
Let $x,y\in\R$ be such that $F\cap (x,y)=\{c\}$.
As $c$ is a limit point of $A$, at least one of $A\cap (x,c)$ or $A\cap (c,y)$ is infinite and discrete;
assume the former (the latter is similar).  
By increasing $x$, we have that $A\cap (x,c)$ is the range of a strictly increasing sequence 
(because $A\cap (x,c)$ has no limit points in $(x,c)$). 
\end{proof}

\begin{rem}
The last paragraph of the proof does not require the group structure.
\end{rem}

We now begin to work toward the proof of Proposition~\ref{main}. 
Our approach will be more model theoretic than our earlier material.
In particular, for the remainder of this section, we use both ``expansion'' and ``reduct'' in the traditional first-order syntactic sense,  ``quantifier-free definable'' means ``defined by a quantifier-free formula of the language under consideration'', and ``QE'' abbreviates ``quantifier elimination''. 

 Let $T_\Q$ be the theory of ordered $\Q$\nbd vector spaces with a distinguished positive element, in the language $L_\Q:=\set{<,+,0,1}\cup \set{\lambda_q: q\in \Q}$, where $0<1$ and $\lambda_q$ is intended to indicate left scalar  multiplication by $q$. 
Let $\grint{\phantom{x}}$ be a new unary function symbol and
put 
$L_\Q'=L_\Q\cup\{\grint{\phantom{x}}\}$.
Let $T_\Q'$ be the union of $T_\Q$ and the universal closures of the following: 
\begin{itemize}
\item
$\grint{\grint{ x}+y}=\grint{ x}+\grint{ y}$
\item
$0\leq x<1\rightarrow \grint{ x}=0$
\item
$\grint{1}=1$
\item
$\grint{ x}\leq x<\grint{ x}+1$
\end{itemize}
(Observe that as $T_\Q$ is universally axiomatizable, so is $T_\Q'$.)
By~\cite{dloivp}*{Appendix},
$T_\Q'$ has QE and is complete.\footnote{Some minor errors in the published proof were noticed and repaired by Trent Ohl while he was a Ph.D.\ student of author Miller.}
It is easy to see that $(\R,<,+,\N)$ is $\emptyset$\nbd interdefinable with 
$$
(\R,<,+,0,1, (q.x)_{q\in\Q},\grint{\phantom{x}})
$$
where 
$\grint{\phantom{x}}$ denotes the usual ``floor'' function $t\mapsto \max(\Z\cap  (-\infty,t])$ for $t\in \R$.
Thus, 
$\Th(\R,<,+,\N)$ is bi-interpretable with $T_\Q'$. 
(Of course, $\Z$ is $\emptyset$\nbd interdefinable with $\N$ over $(\R,<,+)$. It is more convenient to prove Proposition~\ref{main} by working over  $(\R,<,+,\Z)$ instead of $(\R,<,+,\N)$.)
 
Let $\mathfrak Z$ be a structure on $\Z$ in a language $L_\mathfrak{Z}$  such that: 
\begin{itemize}
\item
$L_\mathfrak{Z}$ has no function or constant symbols; 
\item
there is a symbol $Z\in L_{\mathfrak Z}$ that is interpreted as $\Z$ in  $\mathfrak Z$; 
\item
$\Th(\mathfrak Z)$ has QE; 
\item
$\mathfrak Z$ defines $\mid$ (hence also all arithmetic sets).
\end{itemize} 
Put $L=L_\Q'\cup L_\mathfrak Z$ and 
$$
T=T_\Q'\cup \{\forall x\,(Zx \leftrightarrow \grint{x}=x)\}\cup \set{\sigma^Z: \sigma\in\Th(\mathfrak Z)},
$$
where $\sigma^Z$ denotes the relativization of $\sigma$ to $Z$ (see, \eg, Chang and Keisler~\cite{CK} for a definition).  
We will  show that $T$ has QE.
The proof is an extension of that of the corresponding result for $T_\Q'$, but the added complication justifies outlining some details. 
 We begin by disposing of a number of routine technical observations and lemmas (the proofs of which we leave mostly to the reader). 
Unless indicated otherwise, ``term'' means ``$L$\nbd  term''.

Every term is a finite composition of $\grint{\phantom{x}}$ and $L_\Q$\nbd terms.
Every finite composition of  $\lfloor\phantom{x}\rfloor$ is $T$\nbd equivalent to $\grint{\phantom{x}}$.
Every $L_\Q$\nbd term in variables $x_1,\dots,x_n$ is $T_\Q$\nbd equivalent to one of the form 
$\sum_{i=1}^n\lambda_{q_i}x_i+\lambda_{q_0}1$ for some $n\in\N$ and $q_1,\dots,q_n\in\Q$; 
we tend to write 
$
\sum_{i=1}^nq_ix_i+q_0
$
instead of 
$
\sum_{i=1}^n\lambda_{q_i}x_i+\lambda_{q_0}1
$.
If $\tau_1$ and $\tau_2$ are  $L_\Q$\nbd terms in variables $x_1,\dots,x_n$, then there is an $L_\Q$\nbd term $\sigma$ in the same variables 
such that  
\begin{gather*}
T_\Q\vdash \tau_1x_1\dots x_n=\tau_2x_1\dots x_n
\leftrightarrow \sigma x_1\dots x_n=0\\
T_\Q\vdash \tau_1x_1\dots x_n<\tau_2x_1\dots x_n
\leftrightarrow \sigma x_1\dots x_n<0.
\end{gather*}
 
In what follows, we tend to let  $\mathfrak B$ (allowing subscripts) denote an arbitrary model of $T$, with underlying set $B$ (allowing corresponding subscripts).
Let $Z(\mathfrak B)$ denote the interpretation of the symbol $Z$ in $\mathfrak B$. 
Note that $Z(\mathfrak B)=\set{b\in B: \grint{b}=b}$.
We recall a lemma used in the proof of QE of $T_\Q'$, restated for our current purposes, followed by some routine consequences. 

\begin{blank}\label{modlemma}
For all $b\in Z(\mathfrak{B})$ and positive integers $m$ there exists a unique $i\in\{0,\dots, m-1\}$ 
such that  
$
\frac1m(b+i) \in Z(\mathfrak{B})
$.
\end{blank} 

\begin{blank}\label{terms}
Let $n\geq 1$ and $\tau(x_1,\dots,x_n)$ be a term.
\begin{enumerate}
\item
There is a finite partition $\mathcal C$ of $Z(\mathfrak B)$ into sets $\emptyset$\nbd definable in the reduct of $\mathfrak B$ to $L_\mathfrak{Z}$ such that, for each $C\in \mathcal C$\tx, 
there exist $q\in\Q$ and a term $\sigma(x_1,\dots,x_{n-1})$ with 
$\tau=\sigma+qx_n$ on $B^{n-1}\times C$.
\tx(More precisely\tx: The restriction to $B^{n-1}\times C$ of the interpretation of $\tau$ in $\mathfrak B$  is equal to the restriction to $B^{n-1}\times C$
of the interpretation of $\sigma+qx_n$ in $\mathfrak B$.\tx)
\item
There is a finite partition $\mathcal C$ of $Z(\mathfrak B)$ into sets $\emptyset$\nbd definable in the reduct of $\mathfrak B$ to $L_{\mathfrak Z}$ such that, for each $C\in \mathcal C$\tx, 
there exist $q,r\in\Q$ and a term $\sigma(x_1,\dots,x_{n-1})$ such that
for all $(b,c)\in B^{n-1}\times C$\tx, if $\tau(b,c)\in Z(\mathfrak B)$\tx, then $\tau(b,c)=\sigma(b)+qc+r$ and both 
$\sigma(b)$ and $qc+r$ lie in $Z(\mathfrak B)$.
\item
Let $m\leq n$\tx, $b\in B^m$ and $D$ be the definable closure of $b$ 
in the reduct of $\mathfrak B$ to $L_\Q'$. 
There is a finite partition $\mathcal C$ of $[0,1)^{n-m}$ into sets that are $D$\nbd definable in 
the reduct of 
$\mathfrak B$ to $L_\Q$ such that\tx, 
for each $C\in\mathcal C$\tx, 
there exist $q_1,\dots,q_{n-m}\in\Q$ and $d,e\in D$ with
$\tau\rest (\{b\}\times C)=\sum_{i=m+1}^{n}q_ix_i+d$ and $\grint{\tau}\rest (\{b\}\times C)=e$.
\end{enumerate}  
\end{blank}

We omit proofs. 
(Generally, proceed via~\ref{modlemma} and induction on terms, noting that the interesting cases tend to be when $\tau$ is of the form $\grint{\phi}$.) 
 
\begin{blank}\label{zodfbl}
Every subset of $Z(\mathfrak B)^n$ that is quantifier-free definable in $(\mathfrak B, (b)_{b\in B})$ is definable in $(Z(\mathfrak B), (R(\mathfrak B))_{R\in L_\mathfrak Z})$\tx, where $R(\mathfrak B)$ is the interpretation of $R$ in $\mathfrak B$.
\end{blank} 

\begin{proof}[Sketch of proof]
As every arithmetic set is $\emptyset$\nbd definable in $\mathfrak Z$, there is a bijection from 
$Z(\mathfrak B)^n$ to $Z(\mathfrak B)$ that is quantifier-free definable in $\mathfrak B$. 
Thus, it suffices to consider subsets of $Z(\mathfrak B)$ that are atomically defined in $(\mathfrak B, (b)_{b\in B})$. 
Functions given by unary terms of  $(\mathfrak B, (b)_{b\in B})$
are given by compositions of $\grint{\phantom{x}}$ and functions of the form $qx+b$ with $b\in B$.
Employ~\ref{terms}.1 and~\ref{terms}.2 as needed.
\end{proof}

\begin{blank}\label{locdefine}
Let $Y\subseteq B^{m+n}$ be quantifier-free definable in $\mathfrak B$ and $(u,z)\in B^m\times Z(\mathfrak B)^n$.
Let $D$ be the definable closure of $(u,z)$ in the reduct of $\mathfrak B$ to $L_\Q'$.
Then  
$$
\set{x\in B^n: (u,x)\in Y}\cap \prod_{i=1}^n[z_i,z_i+1)
$$
is $D$\nbd definable in the reduct of $\mathfrak B$ to $L_\Q$.
\end{blank}

\begin{proof}
[Sketch of proof]
It suffices to show the result for atomically defined  sets. Employ~\ref{terms}.3.
\end{proof}

\begin{blank}\label{QE}
$T$ is complete and has QE. 
\end{blank}

\begin{proof}
As the $L$\nbd structure $
(\Q,<,+,0,1,(\lambda_q)_{q\in\Q}, \lfloor\phantom{x}\rfloor,\mathfrak Z)
$
embeds into every model of $T$, it is enough to show that $T$ has QE.
We employ one of the standard ``embedding tests''. 
Let $\mathfrak{B}_1,\mathfrak{B}_2\vDash T$ and $\mathfrak{A}$  (with underlying set $A$) be a substructure of $\mathfrak{B}_1$ that is also embedded into $\mathfrak{B}_2$.
Assume that $\mathfrak{B}_2$ is $\operatorname{card}(A)^+$\nbd saturated.
We show that the embedding extends to an embedding of $\mathfrak{B}_1$ into $\mathfrak{B}_2$.

We first reduce to the case that $Z(\mathfrak B_1)=Z(\mathfrak A)$ as follows.
Suppose there exists $c_1\in Z(\mathfrak B_1)\setminus Z(\mathfrak A)$; we show that we may replace 
$\mathfrak A$ with the substructure of $\mathfrak B_1$ generated over $\mathfrak A_1$ by $c_1$.  
(Then iterate as needed.) 
In order to reduce clerical clutter, let us assume that the embedding of $\mathfrak A$ 
into $\mathfrak B_2$ is just the identity on $A$.   
As $\Th(\mathfrak Z)$ has QE, there exists $c_2\in B_2$ such that the $L_\mathfrak Z$\nbd type of $c_2$ over $A\cap Z(\mathfrak{B}_2)$ is equal to the  $L_\mathfrak Z$\nbd type of $c_1$ over $A\cap Z(\mathfrak{B}_1)$.
It suffices now to show that the quantifier-free $L$\nbd type of $c_2$ over $A$ is equal to the quantifier-free $L$\nbd type of $c_1$ over $A$. 
We illustrate by dealing with a special case. 
Let: $m$ and $n$ be positive integers;
$\tau$ be an $n$\nbd ary term; $a\in A^{n-1}$; 
$R$ be an $m$\nbd ary relation symbol of $L_\mathfrak Z$; 
and
$a'\in A^{m-1}$.
If follows from~\ref{terms}.2 applied to each of $\mathfrak B_1$ and $\mathfrak B_2$ that if $\mathfrak B_1\vDash R(\tau(a,c_1),a')$, then $\mathfrak B_2\vDash R(\tau(a,c_2),a')$.
(Other cases are handled similarly using ~\ref{terms}.1 or~\ref{terms}.2.) 
 
Now assume that $Z(\mathfrak B_1)=Z(\mathfrak A)$ and let $c_1\in B\setminus A$. 
As $T_\Q'$ has QE, there exists $c_2\in B_2$ such that the $L_\Q'$\nbd type of $c_2$ over the image of $A$ (under the embedding) is equal to the  $L_\Q'$\nbd type of $c_1$ over $A$.
The rest of the proof is similar to (but easier than) the previous case. 
We omit details. 
 \end{proof}
 
\begin{proof}[Proof of Proposition~\ref{main}]
Let $\rr$ be an expansion of $(\R,<,+,\mid\,)$ by subsets of various $\N^n$. 
By passing to an extension by definitions, we assume that 
$$
\rr=(\R,<,+,0,1,(qx)_{q\in \Q},\grint{\phantom{x}},\mathfrak Z)
$$
where 
$\mathfrak Z$ is the expansion of $\Z$ by all subsets of each $\Z^n$ that are $\emptyset$\nbd definable in $\rr$.
(\emph{Remark:} By~\ref{zodfbl}, every definable subset of any $\Z^n$ is $\emptyset$\nbd definable, and so $\mathfrak Z$ is equal to the structure induced on $\Z$ in $\rr$.) 
By~\ref{QE}, $\Th(\rr)$ has QE and is axiomatized by $T$. 
By QE and~\ref{locdefine}, every bounded set definable in $\rr$ is definable in $(\R,<,+)$, \ie, 
Proposition~\ref{main}.1 holds;
then Proposition~\ref{main}.2 is immediate via~\ref{uniflocominpath}.
It remains to show that $\rr$ is CC. 

Let $E\subseteq \R^n$ be definable in $\rr$.
We must show that each component of $E$ is definable.
Let $m\in\N$, $c\in \R^m$ and $X\subseteq \R^{m+n}$ be $\emptyset$\nbd definable in $\rr$ such that $E=\set{e\in \R^n: (c,e)\in X}$.
Via QE, \ref{zodfbl}, \ref{locdefine} and model-theoretic compactness, there exist
\begin{itemize}
\item
$N\in\N$
\item 
sets $Y_1,\dots,Y_N\subseteq\R^{m+n}$ that are $\emptyset$\nbd definable in $(\R,<,+,1)$
 \item
maps $F_1,\dots,F_N\colon \R^{m+n}\to \R^m$ that are $\emptyset$\nbd definable in  $\rr$ 
\end{itemize}
such that, for each $z\in\Z^n$,
there exists $J\subseteq \{1,\dots,N\}$ such that the sets 
$$
\set{x\in\R^n: (F_j(c,z),x)\in Y_j},\quad j\in J
$$
partition 
$
E\cap \prod_{k=1}^n[z_k, z_k+1)
$
into cells of $(\R,<,+)$.
The set of all $z\in \Z^n$ such that  $\prod_{k=1}^n[z_k, z_k+1)$ intersects $E$ is definable in $\rr$, and  the possible descriptions 
of the cells in  $E\cap \prod_{k=1}^n[z_k, z_k+1)$ can be coded uniformly in $z$ by subsets of $\{0,1\}^d$ for some $d\in\N$. 
Thus, there exist $m\in\N$ and definable $S\subseteq \Z^m\times \R^n$ such that, letting $A$ be the projection of $S$ on the first $m$ coordinates, 
the fibers $S_a$ ($=\set{e\in\R^n: (a,e)\in S}$) of $S$ over $A$ 
partition $E$ into connected sets that are each definable in $(\R,<,+)$.
As any expansion of $(\Z,<,\{0\})$ has Definable Choice, we reduce to the case that 
$a\mapsto S_a$ is injective. 
Let $G$ be the set of all pairs $(a,b)\in A^2$ such that 
$
(S_a\cap \overline{S_b})\cup (S_b\cap \overline{S_a})\neq \emptyset
$;
then $G$ is definable in $\rr$, hence also in $\mathfrak Z$.
As  $\mathfrak Z$ defines all arithmetic sets, each graph component of $(A,G)$ is definable in  $\mathfrak Z$.
(To put this another way: Since $G$ is computable from $A$, each graph component of $(A,G)$ is computably enumerable from $A$, hence definable in $(\Z,+,\cdot, A)$.) 
It is an exercise (\cf~\cite{ominbook}*{Chapter~3, (2.19).5}) to see that $C$ is a graph component of  
$(A,G)$ if and only if $\bigcup_{a\in C}S_a$ is a component of $E$. 
\end{proof}
   
\section{Concluding remarks}

Examination of the proof of Proposition~\ref{main} yields a number of other results; below are three illustrative examples.

\begin{blank}
If $\rr$ is an expansion of  $(\R,<,+,\Z)$ by subsets of various $\Z^n$\tx, then $\Th(\rr)$ is axiomatized 
over the union of $\Th(\R,<,+,1)$ and the relativization to $\Z$ of the theory of structure induced on $\Z$ in $\rr$ 
by expressing that $(\Z,+,1)$ is a substructure of  $(\R,+,1)$ and for every $r\in \R$ there is a unique  
$k\in\Z$ such that $k\leq r<k+1$. 
 \end{blank}

\begin{blank}
If $\rr_0$ is a reduct in the sense of definability of $(\R,<,+)$ over $(\R,<)$\tx, and $\rr$ is an expansion of $(\rr_0,\N)$ by subsets of various $\N^n$\tx, then $\rr$ is CC if and only if it defines~$\mid$\,.  
\end{blank}

(For the QE, replace $\rr_0$ with the expansion in the syntactic sense of $(\R,<,0,1)$ by all $\{0,1\}$\nbd definable functions of $\rr_0$. Every such function is $\emptyset$\nbd definable in $(\R,<,+,1)$.) 

\begin{blank}\label{divgroup}
Let $G$ be a divisible subgroup of $\R$ containing $\Z$ and 
$\mathfrak G$ be an expansion of $(G,<,+,\Z)$ by subsets of various $\Z^n$.
Then $\mathfrak G$ defines $\mid$\, if and only if 
for all $E\subseteq G^n$ definable in $\mathfrak G$ and $x\in E$\tx, the set of all $y\in E$ for which there is a definable-in-$(G,<,+)$ continuous map $\gamma\colon [a,b]\to E$ 
such that 
$\gamma(a)=x$ and $\gamma(b)=y$ is definable in $\mathfrak G$.
\end{blank}

(The appropriate analogue of~\ref{uniflocominpath} follows from~\cite{ominbook}*{p.~100}.) 

\begin{dblank}
Independently of~\cite{dloivp}, an effective proof of QE for 
$\Th(\R,<,+,0,1,(qx)_{q\in \Q},\grint{\phantom{x}})$ 
was produced by Weispfenning~\cite{MR1802076}, though in a slightly different language that did not yield a universal axiomatization.  
Presumably, there is some similar result for $\Th(\R,<,+,\mid\,)$, but relative to an oracle for $\Th(\N,<,\mid\,)$.
We leave further consideration of this matter to the interested reader. 
\end{dblank}

\begin{dblank}
Though currently we are interested primarily in expansions of $(\R,<,+)$ by topological submanifolds, we mention here some potential candidates for PCC structures that are not locally o\nbd minimal but have reasonable model theories.
We suspect that $(\R,<,\Q)$ and  $(\R,<,\Q,-)$ might be PCC, CC and QCC.
(But in any case, $(\R,<,x+1,\Q)$ is neither PCC, CC nor QCC, as it is routine to see that the theory is decidable. Recall Theorem~\ref{maintech}.) 
Let $K$ be the Cantor set.
Is $(\R,<,K)$ PCC, CC or QCC?  (The proof of \cite{dms1}*{7.1} might be useful.) 
Let $E\subseteq \R^2$ be the union of the boundaries of the sets $[-r,r]\times [-r,r]$, $r\in K$.  
We suspect that $(\R,<,E)$ is not CC but every definable set is a union of maximal definably connected sets. 
\end{dblank}
 
\begin{dblank}
If desired, one could say that  a first-order topological structure $\mathfrak{M}$ is ``definably CC'' (or something similar) if each definable set is a union of maximal definably connected subsets, equivalently, 
for all definable $X$ and $x\in X$ there is a maximal definably connected subset of $X$ that contains $x$.
Similarly, one could say that $\mathfrak{M}$ is definably QCC if for all definable $X$ and $x\in X$, the intersection of all clopen definable subsets of $X$ that contain $x$ is definable.
(Every o\nbd minimal structure would then be definably CC and definably QCC, because every definable set would be locally definably connected and a finite union of maximal definably connected subsets.) 
Evidently, a useful analogue of PCC would require more thought, because the definition of ``path connected'' is tied to the real line (but recall~\ref{divgroup} above). 
Note also that by QE of $\Th(\R,<)$, there is no path definable in $(\R,<)$ joining any point in $(-\infty,1)\times (-\infty,2)$ to the point $(1,2)$. 
Thus, in $(\R,<)$, the set $(-\infty,1)\times (-\infty,2)\cup\{(1,2\}$ is path connected and definably connected,  but not ``definably path connected''.  
\end{dblank}

\end{document}